\documentclass[12pt]{elsarticle}

\usepackage{amssymb,amsmath,amsthm}

\DeclareMathOperator{\im}{im}

\newtheorem{theorem}{Theorem}[section]

\newtheorem{corollary}[theorem]{Corollary}
\newtheorem{prop}[theorem]{Proposition}
\newtheorem{remark}[theorem]{Remark}

\theoremstyle{definition}
\newtheorem{definition}[theorem]{Definition}

\newcommand{\id}{\mbox{id}}

\begin{document}
	\title{Big fundamental groups: generalizing homotopy and big homotopy}
		\author{Keith Penrod}
		\address{Department of Mathematics\\ Morehouse College\\ 830 Westview Dr SW\\ Atlanta, GA 30314}
		\ead{keith.penrod@morehouse.edu}
	%	\subjclass[2010]{Primary 55Q05; Secondary 06A05}
		\begin{abstract}
			The concept of big homotopy theory was introduced by J.~Cannon and G.~Conner in \cite{big} using big intervals of arbitrarily large cardinality to detect big loops.  We find, for each space, a canonical cardinal that is sufficient to detect all big loops and all big homotopies in the space.  
		\end{abstract}
		\maketitle

\section{Introduction} % (fold)
\label{sec:introduction}

Cannon and Conner \cite{big} defined big homotopy theory and proved that for any Hausdorff space, the big fundamental group is well-defined.  In this paper, we show that it is well-defined for any pointed space $X$ and calculate a canonical cardinal $\gamma(X)$ that is sufficient for computing the big fundamental group.  In particular, for any cardinal $\alpha$, we define the functor $\pi_1^\alpha$ using big intervals of cardinality at most $\alpha$, show that for any $X$, the big fundamental group $\Pi_1(X)$ is naturally realized by $\pi_1^\alpha$, and even find a canonical interval $I(X)$ such that any big loop in $X$ can be realized as a path from $I(X)$.  

We state here some of the definitions and results from that paper as a convenience to the reader and to provide background and motivation for the main part of the paper, which is to improve upon the theory and generate a family of functors to calculate homotopy in a slightly different way.

The concept of big homotopy theory is simply to extend the idea of a loop to allow for objects other than the real interval $[0,1]$ to be the domain of a loop.  More specifically, a \emph{big interval} is a totally-ordered set which is compact and connected.  Equivalently, a totally-ordered set which is order complete, dense, and has a maximal and minimal element.  A big interval is endowed with the order topology.  Then a big path is merely a continuous function whose domain is a big interval.  A big loop is a big path whose endpoints coincide.  For notational purposes, given a big interval $I$, the minimal element will be denoted $0_I$ and $1_I$ will denote its maximal element.

Big path homotopy is slightly more complicated to define.  Instead of simply a function which starts at one path and ends at the other, we require a slightly weaker criterion.  The big paths $f:I\to X$ and $g:J\to X$ are said to be \emph{big path homotopic} if the there are big intervals $A$ and $B$, a continuous function $H:A\times B\to X$, and surjective order-preserving maps $p:A\to I$ and $q:B\to J$ such that the following are satisfied.
\begin{enumerate}
	\item $f(0_I)=g(0_J)$ and $f(1_I)=g(1_J)$ that is, $f$ and $g$ agree on their endpoints.
	\item $H(s,0_B)=f(p(s))$ for all $s\in A$.
	\item $H(s,1_B)=g(q(s))$ for all $s\in A$.
	\item $H(0_A,t)$ and $H(1_A,t)$ are constant for all $t\in B$.
\end{enumerate}
When such functions exist, we say that the triple $(H,p,q)$ forms a big homotopy from $f$ to $g$.  We think of $B$ as the parameter space and when this space is significant, we call the triple a \emph{$B$-homotopy}.

Big path homotopy is shown to be an equivalence relation.  Given a pointed topological space $(X,x_0)$, the class of big loops based at $x_0$ reduces to a set of equivalence classes of loops using this definition.  In \cite{big}, it is shown that this is a set when $X$ is Hausdorff.  A slight modification of the proof shows that it is true for all $X$.  This set is called the \emph{big fundamental group} of $X$ based at $x_0$ and is denoted by $\Pi_1(X,x_0)$.  This set is turned into a group by defining path concatenation as the group operation and reverse path as the inverse operation.  In this paper, we define for each cardinal $\alpha$ a new group called $\pi_1^\alpha$ and show that the big fundamental group is realized by $\pi_1^\alpha$, for some $\alpha$ depending on $X$.

% section introduction (end)

\section{Big Intervals} % (fold)
\label{sec:big_intervals}
One primary purpose of the paper is to construct, given a cardinal $\alpha$, a functor $\pi_1^\alpha$ which will give a big fundamental group given a pointed space $(X,x_0)$.  For now we will define $\pi_1^\alpha(X,x_0)$ using all big loops in $X$ based at $x_0$ whose domain has density at most $\alpha$, and big homotopies also with parameter spaces of density at most $\alpha$.    We will see below that it is possible to construct one big interval (depending on $\alpha$) which by itself will allow for all the big paths and big homotopies allowed by this definition.

\begin{prop}
	Let $I$ and $J$ be big intervals.  The following are equivalent.
	\begin{enumerate}
		\item There is an order-preserving injection $f:I\to J$
		\item There is a dense subset $A\subset I$ and an order-preserving injection $h:A\to J$.
		\item There is an order-preserving surjection $g:J\to I$
	\end{enumerate}
\end{prop}
\begin{proof}
	Condition (2) follows immediately from condition (1).  We show that (2) implies (3), so assume $h$ exists.  Define $g$ as follows.  For a point $j\in J$, define the set $S_j=\{i\in A\mid h(i)\le j\}$.  Then define $g(j)=\sup S_j$.  Note that for any $j$ such that $S_j=\emptyset$, we have $g(j)=0_I$.  The property of order-preserving is obvious.  To see that it is surjective, let $i\in I$.  Let $j=\sup\{h(i^\prime)\mid i^\prime<i,\ i^\prime\in A\}$.  Since $A$ is dense in $I$, we see that $i=\sup S_{j}$ because for any $i^\prime<i$, it must be that $h(i^\prime)<j$.  Thus $g$ is surjective.
	
	Finally, we show that (3) implies (1).  This relies on the fact that a big interval is order complete.  Suppose that $g$ exists.  Define $f$ by $f(t)=\sup\{s\in J\mid g(s)=t\}$.  Since $g$ is surjective, this set is never empty for any $t$.  The fact that it is order-preserving and injective are obvious.
\end{proof}

It is also noted that if $f:I\to J$ is a function between big intervals which is order-preserving and surjective, then it is automatically continuous.  Using the above result, we can see that if $I$ is a big interval into which other big intervals embed, then it can map onto all of those big intervals, which means that it can be used to represent big loops based on any of those intervals.  In other words, we seek a big interval $I$ that accepts any big interval up to a given cardinality.  

\begin{definition}
	Let $T$ be a totally-ordered set and $\alpha$ a cardinal.  Then we say that $T$ is \emph{$\alpha$-perfect} if the following are true.
	\begin{enumerate}
		\item There is a set $S\subset T$ with cardinality $\alpha$ that is order dense.
		\item Given any totally-ordered set $T^\prime$ with cardinality at most $\alpha$, there is an order-preserving injection $f:T^\prime\to T$.
	\end{enumerate}
\end{definition}

We are interested in particular in the case when $T$ is a big interval.  Thus, any big interval with a dense subset of cardinality at most $\alpha$ being able to embed in $T$ means that $T$ will map onto any such interval.  The other condition---that $T$ have a dense subset of cardinality $\alpha$---is to guarantee that any two such intervals will yield the same group.  Before moving on, we show that such intervals do exist for arbitrarily large $\alpha$ by explicitly constructing them here.

\begin{theorem}
	Let $\alpha$ be an infinite cardinal.  Define $\hat \alpha=\sup\{2^\beta\mid \beta<\alpha\}$.  Then the set $I(\alpha)=[0,1]^\alpha$, endowed with the lexicographical order, is an $\hat \alpha$-perfect big interval.
\end{theorem}
\begin{proof}
	To aid in discussion, we will abbreviate $I(\alpha)$ simply as $I$ since $\alpha$ is clear from context. First we show that $I$ is a big interval.  It is easily seen that the minimal point is the one which is $0$ in every coordinate and the maximal point is the one which is $1$ in every coordinate.  We will show below that $I$ has a subset of cardinality $\hat \alpha$ which is (order) dense.  This implies that $I$ is (order) dense.
	
	To see that $I$ is order complete, let $A\subset I$ be nonempty.  It is necessarily bounded above by the maximal element of $I$.  We construct the element $s$ inductively as follows.  Define $s_0=\sup\{a_0\mid a\in A\}$.  Then, for any $\beta\in \alpha$ for which $s_\gamma$ has been defined for all $\gamma<\beta$, define $s_\beta$ as follows.  Let $A_\beta=\{a\in A\mid a_\gamma=s_\gamma\mbox{ for all }\gamma<\beta\}$.  If $A_\beta$ is empty, define $s_\beta=0$.  Otherwise, define $s_\beta=\sup A_\beta$.  That $s$ defined in this way is the supremum of $A$ is straightforward.  Hence, $I$ is order complete and is therefore a big interval.
	
	To see that $I$ has a dense subset of order $\hat\alpha$, let $S$ be the set of all elements of $I$ which are eventually $0$.  That is, for $s\in S$, there is a coordinate $\beta\in \alpha$ such that $s_\gamma=0$ for all $\gamma\ge \beta$.  Note that for any cardinal $\beta<\alpha$, the set $[0,1]^\beta$ is contained in $S$, identified with its natural image of appending a tail of $0$'s as necessary and that in fact $S=\displaystyle \bigcup_{\beta<\alpha}[0,1]^\beta$. Therefore, we conclude that $\displaystyle |S|=\sup_{\beta\in \alpha}\left\{\left|[0,1]^\beta\right|\right\}$ and for $\beta$ infinite, we have $\left|[0,1]^\beta\right|=\left|\left(2^{\aleph_0}\right)^\beta\right|=\left|2^{\aleph_0\cdot \beta}\right|=\left|2^\beta\right|$, thus $\displaystyle |S|=\sup_{\beta\in \alpha} \left|2^\beta\right|=\hat \alpha$.  To see that $S$ is dense in $I$, let $a<b$.  Let $\beta$ be the minimal coordinate where $a$ and $b$ disagree.  Then there is $t\in [0,1]$ such that $a_\beta<t<b_\beta$.  Thus, the point $s\in S$ given by $s_\gamma=a_\gamma$ for $\gamma<\beta$, $s_\beta=t$, and $s_\gamma=0$ for $\gamma>\beta$ is such that $a<s<b$ as desired.
	
	Finally, we show that if $T$ is any totally-ordered set of cardinality at most $\hat\alpha$, there is an order-preserving injection $f:T\to I$.  Index $T$ by $\hat\alpha$ then define the function $f$ as follows.  First define $f(t_0)=(\frac 12)_\gamma$.  That is, the point that is $\frac 12$ in every coordinate.  Now, given $\beta\in \hat\alpha$ such that $f(t_\gamma)$ is defined for all $\gamma<\beta$, we define $f(t_\beta)$ as follows.  Define $L_\beta^0=\sup\{f(t_\gamma)_0\mid \gamma<\beta \text{ and }t_\gamma<t_\beta\}$ and $U_\beta^0=\inf\{f(t_\gamma)_0\mid \gamma<\beta\text{ and }t_\gamma>t_\beta\}$, allowing for $0$ or $1$ respectively if either set is empty.  It is seen that $L_\beta^0\le U_\beta^0$.  If $L_\beta^0<U_\beta^0$, then define $f(t_\beta)_0=\dfrac{L_\beta^0+U_\beta^0}2$ and $f(t_\beta)_\gamma=\frac 12$ for all $\gamma>0$.
	
	If, however, $L_\beta^0=U_\beta^0$, then for every $\delta\in  \alpha$ such that $L_\beta^\epsilon=U_\beta^\epsilon$ for all $\epsilon<\delta$, define
	\begin{align}
		L_\beta^\delta  &= \sup\{f(t_\gamma)_\delta\mid \gamma<\beta,\ t_\gamma<t_\beta, \text{ and } f(t_\gamma)_\epsilon=L_\beta^\epsilon\text{ for all }\epsilon<\delta \}\\
		U_\beta^\delta &=\inf\{f(t_\gamma)_\delta\mid \gamma<\beta,\ t_\gamma>t_\beta, \text{ and } f(t_\gamma)_\epsilon=L_\beta^\epsilon\text{ for all }\epsilon<\delta\}.
	\end{align}
	
	We claim that there is $\delta\in \alpha$ such that $L_\beta^\delta<U_\beta^\delta$.  Indeed, it is seen that in each coordinate at least $\omega_0$ points will be chosen, using the definition given above, and hence at least $\omega_0^\delta\ge 2^\delta$ points can be mapped using the first $\delta$ coordinates in this method.  Since $\beta<\hat \alpha=\sup\{2^\gamma\mid \gamma<\alpha\}$, we see that there is $\gamma<\alpha$ such that $2^\gamma>\beta$, which implies that the desired coordinate $\delta$ exists.  Let $\delta$ be the minimal coordinate where $L_\beta^\delta<U_\beta^\delta$.  Then define \begin{equation}
		f(t_\beta)_\gamma = \begin{cases}
			L_\beta^\gamma &\text{if }\gamma<\delta\\
			\frac{L_\beta^\delta+U_\beta^\delta}2 &\text{if }\gamma=\delta\\
			\frac 12 &\text{if }\gamma>\delta
		\end{cases}.
	\end{equation}
	
	That $f$ defined in this way is (strictly) order-preserving is obvious.
\end{proof}

A cardinal $\alpha$ is called a \emph{strong limit cardinal} if for any cardinal $\beta<\alpha$ we have $2^\beta<\alpha$.  It is seen that such cardinals of arbitrarily large size exist.  Indeed, $\aleph_0$ is the first such cardinal.  Define $\beth_0=\aleph_0$.  For any ordinal $\gamma$ such that $\beth_{\gamma}$ is defined, define $\beth_{\gamma+1}=2^{\beth_\gamma}$.  Then, for a limit ordinal $\delta$, define $\beth_{\delta}=\displaystyle \bigcup_{\gamma<\delta}\beth_{\gamma}$.  Then for any limit ordinal $\delta$, we see that $\beth_{\delta}$ as defined here is a strong limit cardinal.  

\begin{remark}
	If $\alpha$ is a strong limit cardinal then $\alpha=\sup\{2^\beta\mid \beta<\alpha\}$.  That is, using the notation above, we have $\hat\alpha=\alpha$.  Therefore, $[0,1]^\alpha$, with the lexicographical order, is an $\alpha$-perfect big interval.
\end{remark}

It may be the case that $\alpha$-perfect intervals exist for every cardinal $\alpha$.  The author conjectures that this is the case but is unable to prove it.  However, we have shown that $\alpha$-perfect intervals exist for arbitrarily large $\alpha$, which is sufficient for the purposes of the paper.

% section big_intervals (end)

\section{Big Fundamental Groups} % (fold)
\label{sec:big_fundamental_groups}
Now that we have the intervals necessary, we use them to construct our new big fundamental groups.  The only requirement that we make on our big intervals is that they are $\alpha$-perfect for some $\alpha$.  We define a group on such an interval and then show that it depends only on $\alpha$, not the choice of interval.  

Let $(X,x_0)$ be a pointed space and $I$ a big interval which is $\alpha$-perfect.  Let $\Omega^I_1(X,x_0)$ be the set of all big loops on $I$ based at $x_0$.  That is, the set of all continuous functions $f:I\to X$ with $f(0_I)=f(1_I)=x_0$.  

Loop concatenation is precisely the same as in \cite{big}.  Let $I\vee I$ be the disjoint union of $I$ with itself, identifying the final point of the first set with the initial point of the second.  To dispel ambiguity in discussion, we write this as $I_1\vee I_2$, with $I_1=I_2=I$, simply to distinguish between points in the former and latter copies of $I$.  Let $f,g\in \Omega^I_1(X,x_0)$.   Define $f\vee g:I_1\vee I_2\to X$ by \begin{equation}
	f\vee g(s)=\begin{cases}
		f(s)& \mbox{if }s\in I_1\\
		g(s)& \mbox{if }s\in I_2
	\end{cases}.
\end{equation}

The first thing to note is that we have only define loops whose domain is $I$ and here the domain is $I\vee I$.  However, this is remedied by noting that $I\vee I$ has a dense subset of cardinality $\alpha$, since $I$ does, and therefore there is an order-preserving epimorphism $\phi:I\to I\vee I$.  By pre-composing $f\vee g$ with $\phi$ the domain issue is corrected.  It is obvious from the definition of big homotopy that $(f\vee g)\circ \phi$ is $I$-homotopic to $f\vee g$, and since we are only concerned with homotopy classes, it is unnecessary to make the distinction.

For the inverse operation, given a loop $f\in \Omega_1^I(X,x_0)$, let $\bar I$ denote the set $I$ with the order reversed and define $\bar f:\bar I\to X$ by $\bar f(\bar t)=f(t)$.  Then, as before, since $\bar I$ has a dense subset of cardinality $\alpha$, there is an order epimorphism $\psi:I\to \bar I$ and $\bar f$ is homotopic to $\bar f \circ \psi$.  

It is shown in \cite{big} that if $f\sim f^\prime$ and $g\sim g^\prime$ then $(f\vee g)\sim (f^\prime\vee g^\prime)$ and $\bar f\sim \overline{f^\prime}$.  Let $\pi_1^I(X,x_0)$ be the set of $I$-homotopy classes of loops in $\Omega_1^I(X,x_0)$.  The group structure on $\pi_1^I(X,x_0)$ given by  $[f]\ast [g]=[f\vee g]$ and $[f]^{-1}=[\bar f]$ is the same as in the big fundamental group.  Note, however, that $I$ is the only allowed parameter space for homotopies, whereas in $\Pi_1$, all big homotopies are allowed.

\begin{theorem}
	Let $I$ be an $\alpha$-perfect big interval.  Then for any pointed space $(X,x_0)$ we have $\pi_1^I(X,x_0)\cong \pi_1^\alpha(X,x_0)$.
\end{theorem}
\begin{proof}
	First we observe since $I$ is $\alpha$-perfect, for any big interval $J$ with density at most $\alpha$, there is an order-preserving injection $\nu_J:J\to I$ and hence a surjection $\xi_J:I\to J$.  So, given a loop $f:J\to (X,x_0)$, we define $\varphi:\pi_1^\alpha(X,x_0)\to \pi_1^I(X,x_0)$ byr $\varphi([f]_\alpha)=[f\circ \xi_J]_I$, where $[]_\alpha$ denotes homotopy class via a homotopy whose parameter space has density at most $\alpha$ and $[]_I$ denotes a homotopy whose parameter space is $I$.  
	
	To see that $\varphi$ is well-defined, suppose that $f:J_1\to X$ and $g:J_2\to X$ are $\alpha$-homotopic.  That is, there is a homotopy $H:J\times K\to X$ and surjective order maps $p:J\to J_1$ and $q:J\to J_2$ where $J$ and $K$ each have density at most $\alpha$.  Then we see that $H^\prime:I\times I\to X$ given by $H^\prime(s,t)=H(\xi_J(s),\xi_K(t))$ is an $I$-homotopy from $f\circ \xi_{J_1}$ to $g\circ \xi_{J_2}$, thus $[f]_I=[g]_I$, so $\varphi$ is well-defined. 
	
	To see that $\varphi$ is surjective, let $f:I\to (X,x_0)$ be a loop.  Since $I$ has density $\alpha$, we see that $[f]_\alpha\in \pi_1^\alpha(X,x_0)$ and $\xi_I=\id_I$ and hence $\varphi([f]_\alpha)=[f\circ \id_I]_I=[f]_I$, as desired.
	
	To see that $\varphi$ is injective, suppose that $f,g:I\to (X,x_0)$ are $I$-homotopic loops.  Again, since $I$ has density $\alpha$, this is an $\alpha$-homotopy from $f$ to $g$ hence $[f]_\alpha=[g]_\alpha$ so $\varphi$ is injective.
\end{proof}

We add here one simple example merely to demonstrate that for $\alpha\ne \beta$, the functors $\pi_1^\alpha$ and $\pi_1^\beta$ differ.  For a cardinal $\alpha$, let $I$ be a big interval with a dense set of cardinality $\alpha$ and such that $|I|\le 2^\alpha$.  Let $S(\alpha)$ be $I$ with its endpoints identified.  Now, given a cardinal $\beta$ (not necessarily distinct from $\alpha$), define $B(\alpha, \beta)$ to be the cone of $S(\alpha)$ using an interval $J$ having a dense set of cardinality $\beta$ and such that $|J|\le 2^\beta$.  Then we have the following calculations. \begin{equation}
	\pi_1^\gamma(S(\alpha)) = \begin{cases}
		1 & \mbox{if }\gamma<\alpha\\
		\mathbb{Z} & \mbox{if }\gamma\ge \alpha
	\end{cases}
\end{equation}
\begin{equation}
	\pi_1^\gamma(B(\alpha,\beta)) = \begin{cases}
		1 & \mbox{if } \gamma<\alpha\mbox{ or }\gamma\ge\beta\\
		\mathbb{Z} & \mbox{if }\alpha\le \gamma<\beta
	\end{cases}
\end{equation}

Intuitively, this is because when $\gamma<\alpha$, no loop is big enough to stretch all the way around $S(\alpha)$, but when $\gamma\ge \alpha$ it is and when $\gamma<\beta$, no homotopy is big enough to stretch over the disk bounding $S(\alpha)$ in $B(\alpha,\beta)$, but when $\gamma\ge \beta$, it is.

% section big_fundamental_groups (end)

\section{Realizing the Big Fundamental Group $\Pi_1$} % (fold)
\label{sec:realizing_the_big_fundamental_group_pi_1_}
Since any two $\alpha$-perfect big intervals generate the same group for any given space $X$, we may define the operator $\pi_1^\alpha$ to be $\pi_1^I$ where $I$ is $\alpha$-perfect for any $\alpha$ such that an $\alpha$-perfect big interval exists.  An immediate corollary of the work of Cannon and Conner is that for any pointed space $(X,x_0)$, there is an $\alpha$ such that $\pi_1^\alpha(X,x_0)\cong \Pi_1(X,x_0)$.  We give in this section a slightly improved version of their theorem to give an idea of what cardinality such an $\alpha$ need be.

\begin{definition}
	Given a space $X$ and a point $x\in X$, define the set $N_x$ to be the intersection of every open neighborhood of $x$ in $X$.  Then define the relation $R$ by $xR y$ if $y\in N_x$ and let $\cong$ be the equivalence relation generated by $R$.  That is, $x\cong y$ if there are $x=x_0,x_1,\ldots, x_n=y$ so that for each $i=0,\ldots , (n-1)$ we have either $x_i\in N_{x_{i+1}}$ or $x_{i+1}\in N_{x_i}$.
\end{definition}

Note that the separation axiom $T_1$ is equivalent to the statement that for all $x\in X$, $N_x=\{x\}$, in other words that $\cong$ is the equality relation.  

\begin{prop}
	Let $f,g:X\to Y$ be continuous functions such that for all $x\in X$, we have $f(x)\cong g(x)$ in $Y$.  Then there is a homotopy $H:X\times [0,1]\to Y$ from $f$ to $g$ that fixes every point where $f(x)=g(x)$.  
\end{prop}
\begin{proof}
	Since homotopy is an equivalence relation, it suffices to consider the case when $f(x)\in N_{g(x)}$ for all $x\in X$.  Define the homotopy by \begin{equation}
		H(x,t) = \begin{cases}
			f(x) & \text{if }t<\frac 12\\
			g(x) & \text{if }t\ge\frac 12
		\end{cases}.
	\end{equation}
	It is clear that $H_0=f$, $H_1=g$, and that $H(x,t)$ is constant with respect to $t$ for any $x$ where $f(x)=g(x)$.  All that remains is to show that $H$ is continuous.  It is clearly continuous at the point $(x,t)$ for any $t\ne \frac 12$, so we need only show that for any $x\in X$, $H$ is continuous at $(x,\frac 12)$.  Fix an open set $U$ containing $h(x,\frac 12)=g(x)$.  Since $f(x)\in N_{g(x)}$, we see that $f(x)\in U$.  Since $f$ and $g$ are continuous, we see that there are open neighborhoods $V_1$ and $V_2$ of $x$ such that $f(V_1)\subset U$ and $g(V_2)\subset U$.  Now we note that $(V_1\cap V_2)\times [0,1]$ is an open neighborhood of $(x,\frac 12)$ and that $H((V_1\cap V_2)\times [0,1])\subset f(V_1)\cup g(V_2)\subset U$, so $H$ is continuous at $(x,\frac 12)$.  
\end{proof}

Given a space $X$, let $w(X)$ denote its weight---that is, the minimal cardinal such that $X$ has a basis of size $w(X)$.  Then define the compact weight of $X$ to be $w_c(X)=\sup\{w(C)\mid C\subset X,\ C\text{ compact}\}$.  Note that $w_c(X)\le w(X)$.  Using this, we give a slight modification to the statement and proof of \cite[Theorem 4.20]{big}

\begin{theorem}\label{upper-bound}
	Let $I$ be a big interval and $C$ a compact space.  Given a continuous function $f:I\times C\to X$, there is a big interval $J$ with a dense set of cardinality at most $\alpha=\max\{w_c(X),\aleph_0\}$, a continuous map $g:J\times C\to X$, and an order epimorphism $p:I\to J$ such that for all $(t,u)\in I\times C$, we have $f(t,u)\cong g(p(t),u)$.  
\end{theorem}
\begin{proof}
	Let $\mathcal B$ be a collection of open sets in $X$ whose restriction to $\im f$ is a basis of $\im f$ and such that $| \mathcal B|\le \alpha$.  For every finite subset $\mathcal C$ of $\mathcal B$ that covers $\im f$, let $a_i(\mathcal C), b_i(\mathcal C)\in I$ and $U_i(\mathcal C)\subset C$ be a finite collection with $U_i(\mathcal C)$ open such that $\{(a_i(\mathcal C),b_i(\mathcal C))\times U_i(\mathcal C)\}$ covers $I\times C$ and each set is mapped by $f$ into an element of $\mathcal C$, which is possible since $I\times C$ is compact.  Let $A=\displaystyle \bigcup_{\mathcal C,\ i} \{a_i(\mathcal C),b_i(\mathcal C)\}$.  Since $A$ has only finitely many elements for each finite subset of $\mathcal B$, assuming $\mathcal B$ is infinite, we have that $|A|\le |\mathcal B|\le \alpha$.  If $\mathcal B$ is finite, then $|A|\le \aleph_0\le \alpha$.  If there are any points $t_1<t_2\in A$ such that $(t_1,t_2)\cap A=\emptyset$, embed $\mathbb{Q}$ between $t_1$ and $t_2$, thus making $A$ (linearly) dense.  Note that this does not increase the size of $A$ beyond $\aleph_0\cdot |A|$.  
	
	Define the relation $\sim$ on $I$ by $s\sim t$ if $[s,t]\cap A$ is at most one point.  It is shown in \cite{big} that this is an equivalence relation and that the set $J=I/\sim$ is a big interval with $A$ (identified with its image under $\sim$) as a dense subset.  Define $p:I\to J$ to be the quotient map, namely $p(t)=[t]$.

	Now we show, by contrapositive, that if $s\sim t$ then $f(s,u)\cong f(t,u)$ for all $u\in C$.  Suppose that $s<t$ in $I$ with $f(s,u)\ncong f(t,u)$ for some $u\in C$.  That is, there are open sets $U$ and $V$ in $\im f$ with $f(s,u)\in U\setminus V$ and $f(t,u)\in V\setminus U$.  Let $\mathcal E$ be an open covering of $\im f$.  We see that $\mathcal F=\{E\cap U\mid E\in \mathcal E\}\cup \{V\}$ is also an open covering of $\im f$ and that for all $F\in \mathcal F$, at most one of $f(s,u)$ and $f(t,u)$ is in $F$.  Since $\im f$ is compact, there is a finite refinement $\mathcal B^\prime\subset \mathcal B$ of $\mathcal F$.  By construction of the set $A$, there are $a_1,b_1, a_2,b_2\in A$ and open sets $C_1,C_2\subset C$ with $(s,u)\in (a_1,b_1)\times C_1$ and $(t,u)\in (a_2,b_2)\times C_2$ such that $f((a_1,b_1)\times C_1)\subset B_1$ and $f((a_2,b_2))\times C_2)\subset B_2$ for some $B_1,B_2\in \mathcal B^\prime$.  So we see that $b_1,a_2\in (s,t)\cap A$, therefore $s\nsim t$.  
	
	Now we can define the function $g:(J)\times C\to X$ as follows.  For every equivalence class in $J$, choose a representative $t$ and define $g([t],u)=f(t,u)$.  The discussion in the preceding paragraph shows that, while the definition of $g$ may depend on the representative chosen, for any choice we have that $g([t],u)\cong f(t,u)$ for all $(t,u)\in I\times C$.
	
	To see that $g$ is continuous, let $([t], u)\in J\times C$ and let $U$ be an open neighborhood of $g([t],u)$.  By definition of $g$ there is $s\sim t$ such that $f(s,u)=g([t],u)$ and by construction of the set $A$, there are points $a<s<b$ with $a,b\in A$ and an open neighborhood $B$ of $u$ such that $f((a,b)\times B)\subset U$.  (Of course, there are two other cases, namely that $s<a$ or $s>b$, in which cases replace $(a,b)$ with $[0_I,a)$ or $(b,1_I]$ respectively and the remainder of the proof is the same.)  Since $A$ is (order) dense and $s\sim t$, we see that $([a],[b])$ is a neighborhood of $[t]$, and by definition of $g$ we see that $g(([a],[b])\times B)\subset U$, hence $g$ is continuous at $([t],u)$.  
	
\end{proof}

This result shows that there is an upper bound on how big loops need to be for any particular space $(X,x_0)$ to capture all of the big homotopy information of that space.  This in turn implies that there is a set's worth (rather than a proper class) of big homotopy classes in the construction of $\Pi_1(X,x_0)$.  Now that we have that result, we can sharpen the cardinality given in the above result from $w_c(X)$ to something which, in some cases, may be much smaller.

\begin{definition}
	Given a pointed space $(X,x_0)$, define the cardinality $\alpha(X,x_0)$ as follows.  Given space $I$, let $d(I)$ denote the density of $I$ (that is, the minimal cardinality of a dense set in $I$).  Given a big path homotopy $H:I\times J\to X$, let $\alpha(H)$ be the minimum cardinality among the set $\{\max\{d(I^\prime),d(J^\prime)\}\}$, where $I^\prime$ and $J^\prime$ range over all big intervals such that there is a big path homotopy $H^\prime:I^\prime\times J^\prime\to X$ and order surjections $p:I \to I^{\prime}$ and $q:J \to J^{\prime}$ such that $H(t,u)\cong H^\prime(p(t),q(u))$ for all $(t,u)\in I\times J$.
	
	Then, for each $[f]\in \Pi_1(X,x_0)$, define $\alpha([f])$ to be the minimum cardinality among the set $\{\alpha(H)\}$ where $H$ ranges over all big homotopies from any big loop in $[f]$ to any other.  In other words, if $g$ and $h$ are both big homotopic to $f$ in $X$ and $H$ is a big homotopy from $g$ to $h$, then $\alpha(H)$ is an element of the set in question.  Then we define $\alpha(X,x_0)=\sup\{\alpha([f])\mid [f]\in \Pi_1(X,x_0)\}$.  Theorem \ref{upper-bound} gives that $\alpha(X,x_0)\le w_c(X)$.  
\end{definition}

\begin{corollary}
	If $X$ is (locally) $\alpha$-path connected for some $\alpha$ then $X$ is (locally) $\alpha(X,x_0)$-path connected, where $\alpha(X,x_0)$ is defined as above.  Furthermore, we have that $\pi_1^{\alpha(X,x_0)}(X,x_0)$ is isomorphic to $\Pi_1(X,x_0)$.
\end{corollary}
\begin{proof}
	Let $x,y\in X$.  Since $X$ is $\alpha$-path connected, there is a path $f:I\to X$ from $x$ to $y$ where the density of $I$ is at most $\alpha$.  Then we see that $f\vee \bar f$ is a loop in $\Pi_1(X,x)$.  So, as defined above, there is a loop $g:J\to X$ and an order epimorphism $p:(I\vee \bar I)\to J$ where the density of $J$ is at most $\alpha([g])\le \alpha(X,x_0)$ and such that $(f\vee \bar f)(t)\cong g(p(t))$ for all $t\in I\vee \bar I$.  The ``first half'' of $g$ then provides a path from $x$ to $y$ with a path whose domain has density at most $\alpha(X,x_0)$ as desired.
	
	For the second result, let $f:I\to X$ be a big loop based at $x_0$.  Then we see that there is a path $g:J\to X$ where $J$ has density at most $\alpha([f])\le \alpha(X,x_0)$ and an order epimorphism $p:I\to J$ such that $f(t)\cong g(p(t))$ for all $t\in I$.  Thus every big loop in $(X,x_0)$ can be represented by a loop whose density is at most $\alpha(X,x_0)$.  Furthermore, given any two big homotopic loops $f,g\in X$, there is a homotopy $H:I\times J\to X$ from $f$ to $g$ where $\max\{d(I),d(J)\}=\alpha(H)\le \alpha(X,x_0)$, hence every homotopy can be performed using a homotopy with density at most $\alpha(X,x_0)$, and therefore any two big loops in $(X,x_0)$ which are big homotopic are $\alpha(X,x_0)$-homotopic as well.
\end{proof}

The last result shows given any pointed space $(X,x_0)$, there is an $\alpha$ such that $\pi_1^\alpha(X,x_0)$ captures all the information of big loops that $\Pi_1(X,x_0)$ would capture.  This is why $\Pi_1$ yields a set instead of a proper class.  It can also be seen as a simplification.  For example, if $X$ is a metric space or second countable, then $\Pi_1(X,x_0)\cong \pi_1(X,x_0)$.  Also it is seen by the definition that $\alpha(X,x_0)$ is the smallest cardinality $\alpha$ to guarantee that $\Pi_1(X,x_0)\cong \pi_1^\alpha(X,x_0)$.  For any smaller cardinality, there may be certain loops or homotopies which are not be represented.  In other words, $\pi_1^\alpha(X,x_0)$ may be isomorphic to $\Pi_1(X,x_0)$ for $\alpha<\alpha(X,x_0)$, whereas if $\alpha\ge \alpha(X,x_0)$ then the isomorphism is assured.

Finally, we note that there is a canonical interval which can be thought of as the domain for all loops in a particular space $(X,x_0)$.  In particular, let $\hat\alpha(X,x_0)$ be the least cardinal greater than or equal to $\alpha(X,x_0)$ that can be written in the form $\sup\{2^\beta\mid \beta<\gamma\}$ for some $\gamma$.  Then we see that $I(\hat\alpha(X,x_0))=[0,1]^{\hat\alpha(X,x_0)}$ is an $\hat\alpha(X,x_0)$-perfect big interval and that $\pi_1^{\hat\alpha(X,x_0)}(X,x_0)$ is naturally isomorphic to $\Pi_1(X,x_0)$.  

% section realizing_the_big_fundamental_group_pi_1_ (end)

\nocite{james}
\nocite{lyndon}
\nocite{magnus}
\nocite{turner}
\nocite{eda}
\nocite{eda2}
\nocite{generalized-universal}
\nocite{small}
\nocite{cayley}
\nocite{hrbacek1999introduction}
\nocite{jech2003set}
\nocite{kunen1980set}
\bibliographystyle{amsplain}
\bibliography{infprod.bib}

\end{document}